\documentclass{article}

\usepackage{cite}
\usepackage{amsmath}
\usepackage{amssymb}
\usepackage{amstext}
\usepackage{amsthm}
\usepackage[a4paper, margin=3cm]{geometry}
\usepackage[auth-sc]{authblk}

\numberwithin{equation}{section}
\numberwithin{figure}{section}

\theoremstyle{plain}
\newtheorem{thm}{Theorem}

\theoremstyle{plain}
\newtheorem*{thm*}{Theorem}

\theoremstyle{plain}
\newtheorem{cor}[thm]{Corollary}

\theoremstyle{plain}
\newtheorem*{cor*}{Corollary}

\theoremstyle{plain}
\newtheorem{lem}[thm]{Lemma}

\theoremstyle{plain}
\newtheorem{prop}[thm]{Proposition}

\theoremstyle{definition}
\newtheorem{defn}[thm]{Definition}

\theoremstyle{definition}
\newtheorem{example}[thm]{Example}

\theoremstyle{definition}
\newtheorem{assumption}[thm]{Assumption}

\theoremstyle{remark}
\newtheorem{rem}[thm]{Remark}

\newcommand{\contr}{\mathbin{\lrcorner}}
\DeclareMathOperator{\Diff}{Diff}

\title{Yang--Mills connections on conformally compact manifolds}
\author{Marco Usula\thanks{email: marco.usula@ulb.ac.be}}
\affil{KU Leuven, Université Libre de Bruxelles}

\date{}

\begin{document}

\maketitle

\begin{abstract}
We study a boundary value problem for Yang--Mills connections on bundles over a conformally compact manifold $\overline{M}$. We prove that, for
every Yang--Mills connection $A$ that satisfies an appropriate nondegeneracy
condition, and for every small deformation $\gamma$ of $A_{|\partial\overline{M}}$,
there is a Yang--Mills connection in the interior that extends $A_{|\partial\overline{M}}+\gamma$.
As a corollary, we confirm an expectation of Witten mentioned in his
foundational paper about holography \cite{Witten}.
\end{abstract}


\section{Introduction}

A \emph{conformally compact manifold }is a pair $\left(\overline{M},g\right)$,
consisting of a compact manifold with boundary $\overline{M}$ and
a metric $g$ in the interior $M$ of $\overline{M}$, with the following
property: for some (hence any) boundary defining function $\rho$
for $\overline{M}$, the metric $\rho^{2}g$ extends to a
metric on $\overline{M}$. The prototypical example is the hyperbolic metric, which can be seen as a conformally compact metric on the closed
unit ball $\overline{\mathbb{B}}^{n+1}\subseteq\mathbb{R}^{n+1}$.

Fix an oriented conformally compact manifold $\left(\overline{M},g\right)$,
and a Hermitian vector bundle $\overline{E}\to\overline{M}$. A Hermitian
connection $A$ on $\overline{E}$ is called \emph{Yang--Mills }if it satisfies the equation
\[
d_{A}^{*}F_{A}=0.
\]
This equation is a priori only defined in the interior, because the metric $g$ develops a singularity at $\partial \overline{M}$. However,
there is a way to interpret the equation above as a nonlinear \emph{$0$-elliptic} equation modulo gauge,
defined over the whole compactification (cf.\ Subsection \ref{subsec: 0-calculus}). In this paper, we will present
a proof of a perturbative existence result on Yang--Mills connections,
inspired by a theorem of Graham and Lee on \emph{Poincaré--Einstein
metrics}, which satisfy a different second order nonlinear $0$-elliptic
equation modulo gauge.

A \emph{Poincaré--Einstein} metric on $\overline{M}$ is a conformally
compact metric $g$ on $\overline{M}$ which satisfies the Einstein equation
\[
\text{Ric}_{g}+ng=0.
\]
A conformally compact metric $g$ on $\overline{M}$ induces a conformal
class $\left[(\rho ^2 g)_{|X}\right]$ on the boundary $\partial \overline{M}=X$, called the \emph{conformal
infinity} of $g$. For example, the hyperbolic metric induces the
conformal structure on $S^{n}$ containing the round metric. In the
seminal paper \cite{GrahamLee}, Graham and Lee considered the following
Dirichlet boundary value problem: given a conformal class $\mathfrak{c}$
on $X$, find a Poincaré--Einstein metric $g$ on $\overline{M}$
such that $\left[(\rho ^2 g)_{|X}\right]$. Their main result is
roughly speaking the following: let $g_{0}$ be the hyperbolic metric
on $\mathbb{B}^{n+1}$, and let $\mathfrak{c}_{0}$ be the conformal
class on $S^{n}$ containing the round metric; then, for any conformal
class $\mathfrak{c}$ on $S^{n}$ sufficiently close to $\mathfrak{c}_{0}$,
there is a conformally compact metric $g$ on $\overline{\mathbb{B}}^{n+1}$,
unique modulo diffeomorphisms if sufficiently close to $g_{0}$, whose
conformal infinity is $\mathfrak{c}$. As shown independently by Lee \cite{LeeFredholm} and Biquard \cite{Biquard_Metriques},
this result generalizes to Poincaré--Einstein metrics that satisfy
a certain nondegeneracy condition.

The aim of this paper is to prove the Yang--Mills analogues of the
results just mentioned in the context of Poincaré--Einstein metrics.
We will define affine spaces $\mathcal{A}$ and $\mathcal{A}\left(X\right)$
of Hermitian connections on $\overline{E}$ and $\overline{E}_{|X}$,
modelled on appropriate Banach spaces, and we will define a group
$\mathcal{G}$ of gauge transformations of $\overline{E}$ that restrict
to the identity on the boundary, again modelled on a Banach space.
The main theorem is analougous to Lee's Theorem A in \cite{LeeFredholm} and Theorem I.4.8 in \cite{Biquard_Metriques}. We ignore regularity issues for now; see Theorem \ref{thm: main theorem} for a more precise
statement.
\begin{thm*}Assume that $n+1\geq 4$.
Let $A\in\mathcal{A}$ be a smooth Yang--Mills connection, and suppose
that $A$ satisfies the following nondegeneracy condition: the Laplace-type
operator
\[
L_{A}=d_{A}^{*}d_{A}+d_{A}d_{A}^{*}+\left(-1\right)^{n}\star\left[\cdot\land\star F_{A}\right]
\]
on $\mathfrak{u}\left(E\right)$-valued $1$-forms (here $E=\overline{E}_{|M}$)
has vanishing $L^{2}$ kernel. Then, for every connection $\text{\ensuremath{A_{|X}+\gamma}}$
sufficiently close to $A_{|X}$ in $\mathcal{A}\left(X\right)$, there
is a Yang--Mills connection $B$ in $\mathcal{A}$, unique modulo $\mathcal{G}$ in a neighborhood of $A$, such that $B_{|X}=A_{|X}+\gamma$.
\end{thm*}
As a corollary, we confirm an expectation of Witten mentioned in his
foundational paper about holography \cite{Witten}. This might be
regarded as the Yang--Mills analogue of Graham and Lee's Theorem
A in\textbf{ }\cite{GrahamLee}.
\begin{cor*}Assume that $n+1 \geq 4$.
Suppose that $H^{1}\left(\overline{M},X\right)=0$, and that $\overline{E}=\overline{M}\times\mathbb{C}^{r}$
is the trivial Hermitian vector bundle. Let $A$ be the trivial connection
on $\overline{E}$. Then, for every connection $A_{|X}+\gamma$ sufficiently
close to $A_{|X}$ in $\mathcal{A}\left(X\right)$, there is a Yang--Mills
connection $B$ in $\mathcal{A}$, unique modulo $\mathcal{G}$ in a neighborhood
of $A$, such that $B_{|X}=A_{|X}+\gamma$.
\end{cor*}
The main ideas of the proofs are similar to those used in the previously
mentioned papers about Poincaré--Einstein metrics. However, the technical
tools employed here are different. More precisely, we make use of
the $0$-calculus, introduced by Mazzeo and Melrose in
\cite{MazzeoMelrose}. We hope that his approach will open up possible future developements of Yang--Mills theory on
a large class of complete Riemannian manifolds with ``amenable''
singularities at infinity. More precisely, conformally compact metrics are special examples of complete metrics “with a Lie structure at infinity”, as described in \cite{AmmannLauterNistor_LieStructures}. In all these cases, there is a preferred Lie subalgebra $\mathcal{V}_{\text{Lie}}$ of the Lie algebra of vector fields on $\overline{M}$ tangent to the boundary, satisfying certain properties (cf.\ Definition 3.1 in \cite{AmmannLauterNistor_LieStructures}). Any such Lie algebra $\mathcal{V}_{\text{Lie}}$ comes equipped with special classes of \mbox{$\mathcal{V}_{\text{Lie}}$-metrics}, \mbox{$\mathcal{V}_{\text{Lie}}$-connections}, and $\mathcal{V}_{\text{Lie}}$-differential operators. Moreover, in \cite{CarvalhoNistorQiao_Fredholm}, it is developed a Fredholm theory for \mbox{“$\mathcal{V}_{\text{Lie}}$-elliptic"} operators. It is in general
difficult to find Einstein $\mathcal{V}_{\text{Lie}}$-metrics; on the other hand, flat connections on bundles over $\overline{M}$ always provide examples of Yang--Mills \mbox{$\mathcal{V}_{\text{Lie}}$-connections}. Thus, if one has good control on the model problems at infinity (cf.\ \cite{CarvalhoNistorQiao_Fredholm}), one could hope to be able to deform flat connections to new Yang--Mills \mbox{$\mathcal{V}_{\text{Lie}}$-connections}. However, it is important to remark that the Fredholm theory developed in \cite{CarvalhoNistorQiao_Fredholm} is based on Sobolev spaces, while for applications to nonlinear problems it is often important to be able to work on appropriate Hölder spaces, since they have better multiplication properties. In the particular case of complete edge metrics, the edge calculus of Mazzeo \cite{MazzeoEdgeI} provides this feature.

The paper is organized as follows. In Section \ref{sec: background},
we will recall the material we need about conformally compact manifolds
and $0$-elliptic operators. In Section \ref{sec:Yang=002013Mills-connections},
we will first define appropriate spaces of connections and gauge transformations,
we will present a ``slice theorem'' for the pullback action, and
finally we will prove the main results.

\subsection*{Acknowledgments}

I am grateful to Joel Fine for his continuous support and guidance
during this project, his countless important suggestions, and his
patience in reading the manuscript. I would also like to thank Michael
Singer and Rafe Mazzeo for numerous helpful conversations about this
and related topics, the anonymous referee for his/her helpful comments,
and Yannick Herfray for pointing out Witten's paper \cite{Witten}
to me. During the preparation of this paper, I was supported by the
ERC consolidator grant 646649 \textquotedblleft SymplecticEinstein\textquotedblright .

\section{Geometric and analytic background\label{sec: background}}

Let us fix once and for all an oriented, compact manifold with boundary
$\overline{M}^{n+1}$, with dimension $n+1\geq4$ (the reason for this restriction will be explained in Remark \ref{rem: why n+1>=4}). Denote by $M$
the interior of $\overline{M}$, and by $X$ its boundary.

\subsection{Conformally compact manifolds\label{subsec:Conformally-compact-manifolds}}

Recall that a \emph{boundary defining function} for $\overline{M}$ is a smooth
nonnegative map $\rho:\overline{M}\to\mathbb{R}$ such that $d\rho\not=0$
at each point of $X$, and $\rho^{-1}\left(0\right)=X$.
\begin{defn}
A \emph{conformally compact metric} on $\overline{M}$\emph{ }is a
metric $g$ in the interior $M$ such that for some (hence any) boundary
defining function $\rho$ for $\overline{M}$, the metric $\rho^{2}g$
extends smoothly to a metric on $\overline{M}$.
\end{defn}

\begin{example}
\label{exa:The-standard-example}The standard example of conformally
compact metric is the \emph{hyperbolic metric}: the metric
\[
g_{hyp}=\frac{4dx^{2}}{\left(1-\left|x\right|^{2}\right)^{2}}
\]
on the open unit ball $\mathbb{B}^{n+1}\subseteq\mathbb{R}^{n+1}$
is conformally compact with compactification given by the closed unit ball $\overline{\mathbb{B}}^{n+1}$: the function
$
\rho=\left({1-\left|x\right|^{2}}\right)/2
$
is a boundary defining function, and $\rho^{2}g$ is the Euclidean
metric on $\overline{\mathbb{B}}^{n+1}$.
\end{example}

Let $g$ be a conformally compact metric $g$. Then $g$ is negatively curved ``at infinity'': given a
boundary defining function $\rho$, the function $\left|d\rho\right|_{\rho^{2}g|X}$
is independent of the choice of $\rho$, and for every sequence $\left(p_{k},\pi_{k}\right)$
such that
\begin{align*}
p_{k} & \in M\\
p_{k} & \to p\in X\\
\pi_{k} & \text{ is a \ensuremath{2}-plane in \ensuremath{T_{p_{k}}M}},
\end{align*}
the sequence $\kappa_{g}\left(\pi_{k}\right)$ of the sectional
curvature of $g$ at $\pi_{k}$ converges to $-\left|d\rho\right|^{2}\left(p\right)$
(cf.\ \cite{MazzeoPhd}). For this reason, if $\left|d\rho\right|_{\rho^{2}g|X}\equiv1$,
then $g$ is also called \emph{asymptotically hyperbolic}.

Since it is singular at $X$, $g$ does not have a well defined restriction to the boundary. However, $g$ induces a conformal class on $X$, called the \emph{conformal infinity} of $g$:
\[
\mathfrak{c}_{\infty}\left(g\right):=\left\{ (\rho^{2}g)_{|X}:\rho\text{ is a boundary defining function}\right\} .
\]

Let $\rho$ be a boundary defining function for $\overline{M}$. Then
the flow of $\text{grad}_{\rho^{2}g}\rho$ starting at $X$ induces
a map $X\times[0,\varepsilon)\to\overline{M}$
and, since $d\rho\not=0$ along $X$, for $\varepsilon$ small enough
this map is an embedding. We call such an embedding a \emph{collar}
induced by $\rho$.
\begin{defn}
A\emph{ }boundary defining function $\rho$ for $\overline{M}$ is
called \emph{special} (with respect to $g$) if there exists an $\varepsilon>0$
such that, for every point $p\in X$, the function $\left|d\rho\right|_{\rho^{2}g}$
is constant along the integral curve $[0,\varepsilon)\to\overline{M}$
of $\text{grad}_{\rho^{2}g}\rho$ starting at $p$.
\end{defn}

It is proved in \cite{Graham_Volume}, Lemma 2.1, that for any choice of
$h_{0}\in\mathfrak{c}_{\infty}\left(g\right)$ there exists a (unique
in a neighborhood of the boundary) special boundary defining function
$\rho$ such that $\left(\rho^{2}g\right)_{|X}=h_{0}$. If $g$ is asymptotically
hyperbolic, $\rho$ is special iff $\left|d\rho\right|_{\rho^{2}g}\equiv1$
in a sufficiently small collar $X\times[0,\varepsilon)\to\overline{M}$;
in such a collar, $g$ takes the simple form
\[
g=\frac{d\rho^{2}+h\left(\rho\right)}{\rho^{2}},
\]
where $h\left(\rho\right)$ can be interpreted as a smooth path $[0,\varepsilon)\to C^{\infty}\left(S^{2}\left(T^{*}X\right)\right)$
of metrics on $X$ starting at $h_{0}$.

\subsection{\label{subsec:The--tangent-bundle}The $0$-tangent bundle}

By definition, a conformally compact metric $g$ on $\overline{M}$
can be seen as a smooth section of $S^{2}\left(T^{*}M\right)$ with
a ``double pole'' at infinity. There is an elegant and useful
way to desingularize $g$, by thinking of it as a metric smooth
up to the boundary on a vector bundle on $\overline{M}$ that replaces
$T\overline{M}$. Denote by $\mathcal{V}_0$ the space of smooth vector fields on $\overline{M}$ which vanish along $X$. Elements of $\mathcal{V}_{0}$ are called $0$\emph{-vector fields}.
$\mathcal{V}_{0}$ is a finitely generated and projective module over
$C^{\infty}\left(\overline{M}\right)$ and hence, by the Serre-Swan
Theorem, we can think of $\mathcal{V}_{0}$ as the module of smooth
sections of a vector bundle $^{0}T\overline{M}$, called the $0$\emph{-tangent
bundle} (cf.\ \cite{MazzeoPhd,MazzeoMelrose}). The inclusion
$\mathcal{V}_{0}\to C^\infty\left( T\overline{M} \right)$ induces a bundle map $^{0}T\overline{M}\to T\overline{M}$, and since at each point $p\in M$, the set $\left\{ V_{p}:V\in\mathcal{V}_{0}\right\} $
spans $T_{p}\overline{M}$, this map is an isomorphism in the interior.
Similarly, if we denote by $^{0}T^{*}\overline{M}$ the dual of $^{0}T\overline{M}$,
we have a bundle map $T^{*}\overline{M}\to{^{0}T^{*}\overline{M}}$
which is an isomorphism in the interior. In this language, conformally
compact metrics $g$ on $\overline{M}$ are just bundle metrics on
$^{0}T\overline{M}$, i.e.\ smooth up to the boundary and positive
definite sections of the symmetric power $S^{2}\left(^{0}T^{*}\overline{M}\right)$.
\begin{rem}
\label{Rem: tangent half spaces}Another model for the hyperbolic
space is the \emph{half-space} model. Consider the manifold with boundary
$\overline{\mathbb{H}}^{n+1}=\left\{ \left(t,\xi^{1},...,\xi^{n}\right):t\geq0\right\} \subseteq\mathbb{R}^{n+1}$,
and the following metric in the interior:
\[
g_{hyp}=\frac{dt^{2}+d\xi^{2}}{t^{2}}.
\]
Although
$\overline{\mathbb{H}}^{n+1}$ is noncompact, $^{0}T\overline{\mathbb{H}}^{n+1}$
is still well defined, and we can again interpret $g_{hyp}$ as a
bundle metric on $^{0}T\overline{\mathbb{H}}^{n+1}$. This half-space
model arises naturally on the tangent spaces of an
asymptotically hyperbolic manifold $\left(\overline{M},g\right)$ at points on the boundary. Given $p\in X$, denote by
$\overline{M}_{p}$ the closed inward-pointing half-space of $T_{p}\overline{M}$,
and denote by $M_{p}$ its interior. Choose a boundary defining function
$\rho$ inducing the metric $h_{0}$ on $X$, and choose oriented
normal coordinates $\left(x^{1},...,x^{n}\right)$ for $\left(X,h_{0}\right)$
centered at $p$. Then, in the linear coordinates $\left(t,\xi^{1},...,\xi^{n}\right)$
on $T_{p}\overline{M}$ induced by $\left(\rho,x^{1},...,x^{n}\right)$,
we have
\[
\overline{M}_{p}\equiv\left\{ \left(t,\xi\right):t\geq0\right\} ,
\]
and the globally defined hyperbolic metric on $M_{p}$
\[
g_{p}:=\frac{dt^{2}+d\xi^{2}}{t^{2}}
\]
does not depend on the choices of $\rho$ and $\left(x^{1},...,x^{n}\right)$.
\end{rem}

The exterior powers $^{0}\Lambda^{k}:=\Lambda^{k}\left(^{0}T^{*}\overline{M}\right)$
will be particularly important in this paper; the space of smooth
sections of $^{0}\Lambda^{k}$ over $\overline{M}$ will be denoted
by $^{0}\Omega^{k}$, and its elements will be called \emph{$0$-$k$-forms}.
The bundles $^{0}\Lambda^{k}$ allow us to think of sections of $\Lambda^{k}$
with ``poles at infinity'', as smooth sections up to the boundary
of $^{0}\Lambda^{k}$, just as we did for conformally compact metrics
(this approach is taken, for example, in \cite{MazzeoPhd,MazzeoHodge}).
As a concrete example, take a boundary defining function $\rho$.
Then the $1$-form $d\rho/\rho$ is defined only in the interior,
and it develops a simple pole at infinity. However, we can interpret
it as a globally defined smooth section of $^{0}\Lambda^{1}$. More
generally, any $\omega\in{^{0}\Omega^{k}}$ can be written as $\omega=\rho^{-k}\overline{\omega}$
for some smooth $k$-form $\overline{\omega}$ on $\overline{M}$.

A conformally compact metric $g$ on $\overline{M}$ induces bundle
metrics on each $^{0}\Lambda^{k}$, related to the $\rho^{2}g$ metrics induced on $\Lambda^{k}$ by
$\left|\omega\right|_{g}=\left|\rho^{k}\omega\right|_{\rho^{2}g}$. Moreover, the volume form $\text{dVol}_{g}$ extends uniquely to a
nowhere vanishing section of $^{0}\Lambda^{n+1}$, and the Hodge star
operator extends to a bundle isometry $
\star:{^{0}\Lambda^{k}}\to{^{0}\Lambda^{n+1-k}}$.

Note that $\mathcal{V}_{0}$ is a Lie subalgebra of the Lie algebra of vector fields on $\overline{M}$.
Thanks to this property, one can extend uniquely the exterior differential
$d$ in the interior, and the codifferential $d^*$ with respect to a conformally compact metric, to operators on $0$-differential forms. Recall that the action of $d^*$ on $k$-forms is
\[
d^*=\left(-1\right)^{\left( n+1 \right)\left( k+1 \right)+1 } \star d \star.
\]
This is a good point to set up one more notation. If $a\in{^{0}\Omega^{1}}$,
we will denote by $a \contr \cdot$ the adjoint of the operator $a\land\cdot  $ on $0$-differential forms. Therefore, if $\omega \in {^0 \Omega^k}$, we have
\[
a\contr \omega = \left(-1\right)^{\left(n+1\right)\left(k+1\right)+1}\star\left(a\land\star\omega\right).
\]

\subsection{\label{subsec: zero-connections}$0$-connections}

Let $\overline{E}\to\overline{M}$ be a vector bundle. Denote by $E$
the restriction of $\overline{E}$ to the interior $M$.
\begin{defn}
A \emph{$0$-connection} $A$ on $\overline{E}$ is a map $A:C^{\infty}\left(\overline{E}\right)\to C^{\infty}\left({^{0}\Lambda^{1}}\otimes\overline{E}\right) $
which is scalar-linear and satisfies the Leibniz formula $
A\left(fs\right)=df\otimes s+fAs $
for every $s\in C^{\infty}\left(\overline{E}\right)$, $f\in C^{\infty}\left(\overline{M}\right)$.
\end{defn}

Equivalently, a $0$-connection is a map $
A:\mathcal{V}_{0}\times C^{\infty}\left(\overline{E}\right)\to C^{\infty}\left(\overline{E}\right) $
which is $C^{\infty}\left(\overline{M}\right)$-linear in the first
entry, and such that for every $V\in\mathcal{V}_{0}$ the operator
$A_{V}:C^{\infty}\left(\overline{E}\right)\to C^{\infty}\left(\overline{E}\right)$
satisfies the Leibniz identity.

It follows from the definition that a $0$-connection
on $\overline{E}$ restricts to a connection on
$E$ in the usual sense, and that a ``genuine'' connection on $\overline{E}$ is in
particular a $0$-connection. The converse is not true. In fact, the local $1$-forms representing $A$ with respect to a smooth local frame
of $\overline{E}$ near a point $p\in X$ extend over the boundary to smooth $0$-$1$-form, but they might not extend to genuine $1$-forms. For example, if $g$ is a conformally compact metric on $\overline{M}$, then the Levi-Civita connection on $TM$ does not extend to a genuine connection on $T\overline{M}$, but it extends to a $0$-connection on $^0 T \overline{M}$.

A computation in local coordinates shows that, if $A$ is a $0$-connection
on $\overline{E}$, then the differential $d_{A}$, and the codifferential
$d_{A}^{*}$ with respect to a conformally compact metric on $\overline{M}$,
extend by continuity from the interior to differential operators  on $\mathfrak{gl}\left(\overline{E}\right)$-valued $0$-differential forms. Similarly, the curvature $F_{A}$ of $A$ extends to a $\mathfrak{gl}\left(\overline{E}\right)$-valued
$0$-$2$-form.

In this paper, we will be almost exclusively concerned with ``genuine''
connections on vector bundles. However, it is useful to think
of genuine connections as ``degenerate" examples of \mbox{$0$-connections}. If $A$ is a $0$-connection, with respect to a smooth local frame of $\overline{E}$ defined in a neighborhood in $\overline{M}$ of a $p\in X$, we can write $d_A = d + [a\land \cdot]$, where $a$ is a $\mathfrak{gl}(\overline E)$-valued $0$-$1$-form. If $A$ is a genuine connection, then these $0$-$1$-forms all vanish along the boundary, so $d_A$ is "asymptotic" to the \emph{untwisted} exterior differential $d$. This property will be crucial in Subsection \ref{subsec:Analytic-preliminaries}.

\subsection{\label{subsec: 0-calculus}$0$-differential operators}

In this subsection, we will recall part of the elliptic theory\emph{
$0$-differential operators}, developed by Mazzeo and Melrose in \cite{MazzeoMelrose},
and then refined and generalized in \cite{MazzeoEdgeI,MazzeoEdgeII}.
A different successful approach to the same type of operators is developed
in \cite{LeeFredholm}. For reasons of brevity, we will adopt a
black box approach, and only present here the definitions and the
theorems needed. We refer to \cite{MazzeoEdgeI} for a complete
and detailed description of this theory.

Let $\overline{E},\overline{F}$ be vector bundles on $\overline{M}$.
Denote by $E,F$ the restrictions of $\overline{E},\overline{F}$
to the interior $M$.
\begin{defn}
\label{Def: local expression 0-diff}A $P\in\Diff^{m}\left(\overline{E},\overline{F}\right)$
is called a $0$\emph{-differential operator} if, near any point $p\in X$,
we can write
\[
P=\sum_{j+\left|\beta\right|\leq m}P_{j,\beta}\left(\rho,x\right)\left(\rho\partial_{\rho}\right)^{j}\left(\rho\partial_{x}\right)^{\beta}
\]
with respect to a choice of a boundary defining function $\rho$,
a chart $\left(x^{1},...,x^{n}\right)$ of $X$ centered at $p$,
and trivializations of $\overline{E},\overline{F}$ near $p$. The
space of $0$-differential operators of order $m$ from sections of
$\overline{E}$ to sections of $\overline{F}$ is denoted by $\Diff_{0}^{m}\left(\overline{E},\overline{F}\right)$.
\end{defn}

\begin{example}
It is easy to check that the exterior differential $d$
is in $\Diff_{0}^{1}\left({^{0}\Lambda^{k}},{^{0}\Lambda^{k+1}}\right)$,
and that the codifferential $d^{*}$ with respect to a conformally compact
metric on $\overline{M}$ is in $\Diff_{0}^{1}\left({^{0}\Lambda^{k+1}},{^{0}\Lambda^{k}}\right)$.
More generally, if $A$ is a $0$-connection on $\overline{E}$, then
$d_{A}$ and $d_{A}^{*}$ are $0$-differential operators on $\mathfrak{gl}\left(\overline{E}\right)$-valued
$0$-differential forms.
\end{example}

We will now define some Banach spaces of functions and sections of
vector bundles on $\overline{M}$, naturally associated to the Lie
algebra $\mathcal{V}_{0}$. Choose an auxiliary conformally compact
metric $g$ on $\overline{M}$. We denote by $L^2_0$ and $C^{0,\alpha}_0$ the standard Banach spaces of functions associated to the complete Riemannian manifold $\left(M,g\right)$ (note that conformally compact metrics have positive injectivity radius). Now, choose an auxiliary boundary
defining function $\rho$, and define
\begin{align*}
\rho^{\delta}L_{0}^{2,k} & =\left\{ \rho^{\delta}u:V_{1}\cdots V_{l}u\in L_{0}^{2},\forall V_{j}\in\mathcal{V}_{0},l\leq k\right\} \\
\rho^{\delta}C_{0}^{k,\alpha} & =\left\{ \rho^{\delta}u:V_{1}\cdots V_{l}u\in C_{0}^{0,\alpha},\forall V_{j}\in\mathcal{V}_{0},l\leq k\right\} .
\end{align*}
To make these spaces well defined, we can embed $L_{0}^{2}$ and $C_{0}^{0,\alpha}$
in the space $C^{-\infty}\left(\overline{M}\right)$ of extendible
distributions, i.e.\ the dual (with respect to the topology of uniform convergence of all derivatives) of the space $\dot{C}^{\infty}\left(\Lambda\right)$
of smooth densities on $\overline{M}$ vanishing to infinite order
along $X$. The derivatives above are then well defined weakly.

We can define norms on $\rho^{\delta}L_{0}^{2,k}$ and $\rho^{\delta}C_{0}^{k,\alpha}$
by choosing an auxiliary $0$-connection $\nabla$ on $^{0}T\overline{M}$
(for example the Levi-Civita $0$-connection induced by $g$) and
defining
\begin{align*}
\left|\left|u\right|\right|_{\rho^{\delta}L_{0}^{2,k}}^{2} & =\left|\left|\rho^{-\delta}u\right|\right|_{L_{0}^{2,k}}^{2}=\sum_{j=0}^{k}\left|\left|\nabla^{j}\left(\rho^{-\delta}u\right)\right|\right|_{L_{0}^{2}}^{2}\\
\left|\left|u\right|\right|_{\rho^{\delta}C_{0}^{k,\alpha}} & =\left|\left|\rho^{-\delta}u\right|\right|_{C_{0}^{k,\alpha}}=\sum_{j=0}^{k}\sup_{p\in M}\left|\left(\nabla^{j}\left(\rho^{-\delta}u\right)\right)_{p}\right|+\left[\nabla^{k}\left(\rho^{-\delta}u\right)\right]_{\alpha}.
\end{align*}
Here $\nabla^{j}\left(\rho^{-\delta}u\right)$ is interpreted as a
section (in the interior) of the bundle $\left(^{0}\Lambda^{1}\right)^{\otimes^{j}}$
equipped with the bundle metric induced by $g$, and
\[
\left[\nabla^{k}v\right]_{\alpha}=\sup_{\begin{smallmatrix}p\not = q\in M\\
d_{g}\left(p,q\right)<\iota_{g}
\end{smallmatrix}}\frac{\left|T_{p\to q}\left(\nabla^{k}v\right)_{p}-\left(\nabla^{k}v\right)_{q}\right|}{d_{g}\left(p,q\right)^{\alpha}},
\]
where $\iota_{g}$ is the positive injectivity radius of $g$ and  $T_{p\to q}$ is the parallel transport of the $0$-connection
(that is, the parallel transport of the restriction of the $0$-connection
to the interior) induced by $\nabla$ on $\left(^{0}\Lambda^{1}\right)^{\otimes^{k}}$
along the shortest geodesic connecting $p$ to $q$. Similarly, if $\overline{E}$ is a vector bundle on $\overline{M}$,
we can define the weighted spaces $\rho^{\delta}L_{0}^{2,k}\left(\overline{E}\right)$ and $\rho^{\delta}C_{0}^{k,\alpha}\left(\overline{E}\right)$
by means of auxiliary choices of $\rho,g,\nabla$, a bundle metric
$\left\langle ,\right\rangle _{\overline{E}}$ and a $0$-connection
$\nabla^{\overline{E}}$.

As usual, the norms we just defined do depend on the choices made.
However, the compactness of $\overline{M}$ and the fact that two
$0$-connections on $\overline{E}$ differ by a $\mathfrak{gl}\left(\overline{E}\right)$-valued
$0$-$1$-form, which is bounded in the metric induced by $g$ and
$\left\langle ,\right\rangle _{\overline{E}}$, implies that different
choices lead to equivalent norms. Therefore, the topologies on $\rho^{\delta}L_{0}^{2,k}\left(\overline{E}\right),\rho^{\delta}C_{0}^{k,\alpha}\left(\overline{E}\right)$
are well defined.

An alternative way to define the same topologies is to use an appropriately
chosen set of local charts of $\overline{M}$ where the bundles involved
are trivialized. This local approach is taken for example in \cite{LeeFredholm},
where the following useful results are proved (Lemmas 3.6 and 3.7 in \cite{LeeFredholm}). Denote by $C^{k}\left(\overline{M}\right)$
and $C^{k,\alpha}\left(\overline{M}\right)$ the standard Banach
spaces on $\overline{M}$.
\begin{lem}
\label{lem: (the useful lemma)}Let $k,k'\in\mathbb{N}$, $\alpha\in\left(0,1\right)$,
and $\delta,\delta'\in\mathbb{R}$.
\begin{enumerate}
\item We have a bounded inclusion
\[
C^{k,\alpha}\left(\overline{M}\right)\hookrightarrow C_{0}^{k,\alpha}.
\]
More generally, we have a bounded inclusion
\[
C^{k,\alpha}\left(\Lambda^{k}\right)\hookrightarrow\rho^{k}C_{0}^{k,\alpha}\left({^{0}\Lambda^{k}}\right).
\]
\item If $\delta>0$, then we have a bounded inclusion
\[
\rho^{\delta}C_{0}^{0,\alpha}\hookrightarrow C^{0}\left(\overline{M}\right).
\]
\item If $k\geq k'$ and $\delta\geq\delta'$, then we have a
bounded inclusion
\[
\rho^{\delta}C_{0}^{k,\alpha}\hookrightarrow\rho^{\delta'}C_{0}^{k',\alpha},
\]
which is compact if $\delta>\delta'$ and $k>k'$.
\item The multiplication map
\[
\rho^{\delta}C_{0}^{k,\alpha}\times\rho^{\delta'}C_{0}^{k,\alpha}\to\rho^{\delta+\delta'}C_{0}^{k,\alpha}
\]
is continuous.
\end{enumerate}
\end{lem}

Let $P\in\Diff_{0}^{m}\left(\overline{E},\overline{F}\right)$. Then, directly from the definitions, we have boundedness between appropriate $0$-Sobolev and $0$-Hölder spaces:
\begin{prop}
For every $\delta\in\mathbb{R}$, $k\in\mathbb{N}$, $\alpha\in\left(0,1\right)$,
the following maps are bounded:
\begin{align}
P & :\rho^{\delta}L_{0}^{2,k+m}\left(\overline{E}\right)\to\rho^{\delta}L_{0}^{2,k}\left(\overline{F}\right)\label{eq: P on sobolev}\\
P & :\rho^{\delta}C_{0}^{k+m,\alpha}\left(\overline{E}\right)\to\rho^{\delta}C_{0}^{k,\alpha}\left(\overline{F}\right).\label{eq: P on Hölder}
\end{align}
\end{prop}

We will now discuss the \emph{Fredholm} properties of these maps.
We will see that Fredholmness depends essentially on two microlocal
models associated to $P$. The first one is the \emph{principal \mbox{$0$-symbol}},
which is defined at each point of $\overline{M}$.
\begin{defn}
The \emph{principal $0$-symbol }of $P$ is the unique smooth section
$^{0}\sigma_{P}$ of the bundle $S^{m}\left(^{0}T\overline{M}\right)\otimes\hom\left(\overline{E},\overline{F}\right)$
extending the usual principal symbol of
$P$ from the interior. It is characterized by the induced fibrewise homogeneous of
degree $m$ map
\[
^{0}\sigma_{P}:{^{0}\Lambda^{1}}\to\hom\left(\overline{E},\overline{F}\right).
\]
$P$ is called $0$\emph{-elliptic }if, for every $\xi\in{^{0}\Lambda^{1}}$
nonzero, $^{0}\sigma_{P}\left(\xi\right)$ is invertible.
\end{defn}

\begin{example}
The Laplacian $\Delta$ with respect to a conformally compact metric
$g$ on $\overline{M}$ is $0$-elliptic with principal $0$-symbol $\alpha\mapsto-\left|\alpha\right|^{2}$.
The same is true for the Hodge Laplacians $d_{A}^{*}d_{A}+d_{A}d_{A}^{*}$
on $\mathfrak{gl}\left(\overline{E}\right)$-valued $0$-$k$-forms
associated to any $0$-connection $A$ on $\overline{E}$. 
\end{example}

Suppose that $P$ is $0$-elliptic. On a closed manifold, elliptic
operators induce Fredholm maps on appropriate Sobolev and Hölder spaces.
However, $0$-ellipticity is not enough to obtain Fredholmness of
the maps \ref{eq: P on sobolev}, \ref{eq: P on Hölder}:
\begin{example}
The Dirichlet boundary
value problem on hyperbolic space
\[
\begin{cases}
\Delta u=0 & u\in C^{\infty}\left(\overline{\mathbb{B}}^{n+1}\right)\\
u_{|S^{n}}=f & f\in C^{\infty}\left(S^{n}\right)
\end{cases}
\]
has a unique solution for every $f\in C^{\infty}\left(S^{n}\right)$.
The solutions lie in $\rho^{\delta}C_{0}^{k,\alpha}$ and $\rho^{\delta-\frac{n}{2}}L_{0}^{2,k}$
for every $\delta<0$ and $k\in\mathbb{N}$, so $\Delta$ has an infinite-dimensional
kernel on those spaces.
\end{example}

However, $0$-ellipticity is enough for $0$-elliptic regularity.
The following proposition is a consequence of Theorem 3.8 in \cite{MazzeoEdgeI}.
\begin{prop}
($0$-elliptic regularity) Let $v\in\rho^{\delta}L_{0}^{2,k}$
(resp. $v\in\rho^{\delta}C_{0}^{k,\alpha}$),
and let $u\in\rho^{\delta}L_{0}^{2}$ (resp.
$u\in\rho^{\delta}C_{0}^{0,\alpha}$) be
a weak solution of $Pu=v$. Then $u\in\rho^{\delta}L_{0}^{2,k+m}$
(resp. $u\in\rho^{\delta}C_{0}^{k+m,\alpha}$).
\end{prop}

In order to obtain Fredholm properties, we need to study an additional
model for $P$ at each point $p \in X$: the \emph{normal operator} $N_p\left(P\right)$.
\begin{defn}
The \emph{normal operator} $N_p\left(P\right)$ is a $0$-differential
operator on $\overline{M}_p$ (cf.\ Remark \ref{Rem: tangent half spaces}) between sections of the trivial bundles $\overline{M}_p\times \overline{E}_p$ and $\overline{M}_p\times \overline{F}_p$, defined in terms of
the local expression of $P$ as in Definition \ref{Def: local expression 0-diff}
by
\[
N_{p}\left(P\right)=\sum_{j+\left|\beta\right|\leq m}P_{j,\beta}\left(0,0\right)\left(t\partial_{t}\right)^{j}\left(t\partial_{\xi}\right)^{\beta}.
\]
\end{defn}
It can be checked that the definition above is well posed, i.e.\ it
does not depend on the choice of the coordinates and the trivializations
of the bundles.

A further reduction of the normal operator is the \emph{indicial operator}
of $P$.
\begin{defn}
The \emph{indicial operator} of $P$ is the family $I_{s}\left(P\right)$ of sections of $\text{\ensuremath{\hom}\ensuremath{\left(\overline{E}_{|X},\overline{F}_{|X}\right)}}$
parametrized by $s\in\mathbb{C}$, defined as
\[
I_{s}\left(P\right):=\left(\rho^{-s}P\rho^{s}\right)_{|X}
\]
where $\rho$ is a boundary defining function for $\overline{M}$.
A number $\mu\in\mathbb{C}$ is called an \emph{indicial root} for
$P$ at $p\in X$ if $I_{\mu}\left(P\right)\left(p\right):\overline{E}_{p}\to\overline{F}_{p}$
is not invertible. In terms of the local expression as in Definition
\ref{Def: local expression 0-diff}, the indicial operator is
\[
I_{s}\left(P\right)\left(p\right)=\sum_{j\leq m}P_{j,0}\left(0,0\right)s^{j}.
\]
\end{defn}

Directly from the definition, we have $I_{s}\left(P\right)\left(p\right)=I_{s}\left(N_{p}\left(P\right)\right)$,
so the indicial roots of $P$ at $p\in X$ are exactly the (constant)
indicial roots of $N_{p}\left(P\right)$. Moreover, $I_{s}\left(P\right)\left(p\right)$
is a polynomial in $s$ with coefficients in $\hom\left(\overline{E}_{p},\overline{F}_{p}\right)$,
and with invertible leading coefficient (it is the principal $0$-symbol
of $P$ evaluated at $\left(d\rho/\rho\right)_{p}$), so at each
$p\in X$ there are only finitely many indicial roots.

The following proposition summarizes some useful algebraic properties of these model operators.
\begin{prop}
\label{prop: composition properties of normal ops}Let $\overline{U},\overline{V},\overline{W}$
be bundles over $\overline{M}$. Let $P\in\Diff_{0}^{k}\left(\overline{U},\overline{V}\right)$.
\begin{enumerate}
\item The formal adjoint $P^{*}$ with respect to a conformally compact
metric $g$ on $\overline{M}$ and bundle metrics on $\overline{U},\overline{V}$
is an element of $\Diff_{0}^{k}\left(\overline{V},\overline{U}\right)$,
and
\[
\begin{array}{rll}
^{0}\sigma_{P^{*}}\left(\alpha\right) & =\left(^{0}\sigma_{P}\left(\alpha\right)\right)^{*} & \forall\alpha\in{^{0}\Lambda^{1}}\\
N_{p}\left(P^{*}\right) & =N_{p}\left(P\right)^{*} & \forall p\in X\\
I_{s}\left(P^{*}\right) & =I_{n-\overline{s}}\left(P\right)^{*} & \forall s\in\mathbb{C}.
\end{array}
\]
\item Let $Q\in\Diff_{0}^{h}\left(\overline{V},\overline{W}\right)$. Then
$Q\circ P\in\Diff_{0}^{k+h}\left(\overline{U},\overline{W}\right)$,
and
\[
\begin{array}{rll}
^{0}\sigma_{Q\circ P}\left(\alpha\right) & ={^{0}\sigma_{Q}}\left(\alpha\right)\circ{^{0}\sigma_{P}}\left(\alpha\right) & \forall\alpha\in{^{0}\Lambda^{1}}\\
N_{p}\left(Q\circ P\right) & =N_{p}\left(Q\right)\circ N_{p}\left(P\right) & \forall p\in X.
\end{array}
\]
\end{enumerate}
\end{prop}

We will now restrict to \emph{formally self-adjoint} operators. Assume that $\overline{E}$
has a bundle metric. Let $P\in\Diff_{0}^{m}\left(\overline{E}\right)$
be $0$-elliptic and formally self-adjoint. To use the Fredholm results of \cite{MazzeoEdgeI}, we need an
additional assuption on $P$, which is valid for many geometric operators
(and in particular for the operators we are interested in in this
paper):
\begin{assumption}
\label{hyp: constancy indicial roots}The
indicial roots of $P$ do not depend on $p \in X$.
\end{assumption}

From Proposition \ref{prop: composition properties of normal ops}, the indicial roots of $P$ are symmetric about the line $\Re\left(s\right)=n/2$ in $\mathbb{C}$. Suppose that there are no indicial roots with real part equal to $n/2$, and denote by $\left(\delta_{-},\delta_{+}\right)$
the maximal interval containing $n/2$ for which there are no indicial roots $s$ with $\Re\left(s\right)\in \left(\delta_{-},\delta_{+}\right)$. Note that this interval must be symmetric about $n/2$. The Fredholm theorem we need follows from Theorem 6.1 and Proposition 7.17 in \cite{MazzeoEdgeI}\footnote{There is a discrepancy between the weights in the $0$-Sobolev spaces appearing in \cite{MazzeoEdgeI}
and those appearing here. This is due to the following different conventional
choices. In \cite{MazzeoEdgeI}, the $L^{2}$ norms are all computed
with respect to a reference density \emph{smooth up to the boundary},
while we use the $0$-volume form $\text{dVol}_{g}$ associated to
a conformally compact metric $g$ on $\overline{M}$; since $\text{dVol}_{g}=\rho^{-\left(n+1\right)}\nu$
for some smooth volume form $\nu$ on $\overline{M}$, we have $L^{2}\left(\text{dVol}_{g}\right)=\rho^{\frac{n+1}{2}}L^{2}\left(\nu\right)$.}. See also Theorem C in \cite{LeeFredholm}.
\begin{thm}
\label{thm: Mazzeo's thm self-adjoint}(Mazzeo) Suppose that, for every $p \in X$, the normal operator $N_p \left(P\right)$ is invertible as a map $L^{2,m}_0\left(\overline{E}_p\right) \to L^2_0 \left(\overline{E}_p\right)$. Then, for every
$\delta\in\left(\delta_{-},\delta_{+}\right)$, $k\in\mathbb{N}$
and $\alpha\in\left(0,1\right)$, the maps
\begin{align}
P & :\rho^{\delta-\frac{n}{2}}L_{0}^{2,k+m}\left(\overline{E}\right)\to\rho^{\delta-\frac{n}{2}}L_{0}^{2,k}\left(\overline{E}\right)\label{eq: fredholm sobolev}\\
P & :\rho^{\delta}C_{0}^{k+m,\alpha}\left(\overline{E}\right)\to\rho^{\delta}C_{0}^{k,\alpha}\left(\overline{E}\right)\label{eq: fredholm Hölder}
\end{align}
are Fredholm of index zero. Moreover, they all have the same kernel.
\end{thm}

\begin{rem}
The proof that the index of the maps \ref{eq: P on sobolev} and \ref{eq: P on Hölder} is zero, for a given fixed weight $\delta \in \left(\delta_-,\delta_+\right)$, works exactly as in the closed case using the generalized inverse $G$ and orthogonal projectors $\Pi_1$ and $\Pi_2$ of $P$ as an unbounded operator on $\rho^{\delta - \frac{n}{2}}L^2_0$, with respect to the Hilbert product induced by a choice of $\rho$. Using the mapping properties of $G,\Pi_1 , \Pi_2$ proved in \cite{MazzeoEdgeI}, one can prove that the range of \ref{eq: fredholm sobolev} (resp. \ref{eq: fredholm Hölder}) is complemented by the kernel of the $\rho^{\delta-\frac{n}{2}}L^{2}_0 $ formal adjoint $P^\dagger$ in $\rho^{\delta-\frac{n}{2}}L^{2}_0 $ (resp. $\rho^{\delta}C_{0}^{0,\alpha}$). $P^\dagger$ is related to $P^* = P$ by $P^\dagger = \rho^{2 \delta - n} P \rho^{n - 2\delta}$, so this kernel is isomorphic to the $\rho^{\left(n-\delta\right)-\frac{n}{2}}L^{2,m}_0 $ (resp. $\rho^{n-\delta}C_{0}^{0,\alpha}$) kernel of $P$. Since the interval $\left(\delta_{-},\delta_{+}\right)$ is symmetric about $n/2$, it contains $n-\delta$ as well, so this kernel is equal to the common kernel of the maps \ref{eq: fredholm sobolev} and \ref{eq: fredholm Hölder}. It follows that the index is zero.
\end{rem}

\begin{rem}
Thanks to the hypothesis of self-adjointness, in Theorem \ref{thm: Mazzeo's thm self-adjoint} it is sufficient to assume that $N_p \left(P\right)$ has vanishing $L^2_0$ kernel. If this holds, then using \cite{MazzeoEdgeI} one can prove exactly as above that the range is orthogonally complemented by the $L^2_0$ kernel, so the cokernel vanishes as well.
\end{rem}

\section{Yang--Mills Connections\label{sec:Yang=002013Mills-connections}}

Let us fix once and for all a conformally compact metric $g$ on $\overline{M}$.
For computational convenience, we assume that $g$ is asymptotically
hyperbolic, but there are no conceptual difficulties in extending
the results of this paper to general conformally compact metrics.
Let us also fix a Hermitian vector bundle $\overline{E}\to\overline{M}$
of complex rank $r$, and denote by $E$ its restriction to the interior
$M$. The metric on $\mathfrak{u}\left(\overline{E}\right)$ will
be the one induced by the linear metric $\left\langle A,B\right\rangle \mapsto-\text{tr}\left(AB\right)$
on $\mathfrak{u}\left(r\right)$.

\subsection{\label{subsec:Analytic-preliminaries}Analytic preliminaries}

Let $A$ be a smooth connection on $\overline{E}$. The results of
this paper rely on the Fredholm properties of two Laplace-type operators
associated to $A$:
\begin{align*}
\Delta_{A}:=d_{A}^{*}d_{A} & \in\text{Diff}_{0}^{2}\left(\mathfrak{u}\left(\overline{E}\right)\right)\\
L_{A}:=d_{A}^{*}d_{A}+d_{A}d_{A}^{*}+\left[\cdot^{*}\contr F_{A}\right] & \in\text{Diff}_{0}^{2}\left(^{0} \Lambda_{\mathfrak{u}\left(\overline{E}\right)}^{1}\right).
\end{align*}
$\Delta_A$ and $L_A$ are both $0$-elliptic and formally self-adjoint (the self-adjointness of the bundle map $a\mapsto \left[a^{*}\contr F_{A}\right] = \left(-1\right)^n \star \left[a \land \star F_{A}\right]$ on $\mathfrak{u}\left(\overline{E}\right)$-valued $0$-$1$-forms follows directly from the definition of $\left[\cdot \land \cdot \right]$ and our choice of the metric on $\mathfrak{u}\left(\overline{E}\right)$). Let's now determine their normal operators.
Given $p\in X$, choose a boundary defining function $\rho$ inducing
a metric $h_{0}=\left(\rho^{2}g\right)_{|X}$ on $X$, normal coordinates
$\left(x^{1},...,x^{n}\right)$ for $\left(X,h_{0}\right)$ centered
at $p$, a unitary frame for $\overline{E}$ near $p$, and an arbitrary
basis of $\mathfrak{u}\left(r\right)$. These choices allow us to
interpret the normal operator $N_{p}\left(\Delta_{A}\right)$ (resp.
$N_{p}\left(L_{A}\right)$) as an operator on $\mathbb{R}^{r^{2}}$-valued
functions (resp. $\mathbb{R}^{r^{2}}$-valued $0$-$1$-forms) over
$\overline{\mathbb{H}}^{n+1}$.
\begin{lem}\label{lem: normal ops of L_A}
The normal operator $N_{p}\left(\Delta_{A}\right)$ (resp. $N_{p}\left(L_{A}\right)$)
is the direct sum of $r^{2}$ copies of the Hodge Laplacian $\Delta_{0}$
(resp. $\Delta_{1}$) on functions (resp. $0$-$1$-forms) over $\overline{\mathbb{H}}^{n+1}$.
\end{lem}

\begin{proof}
With respect to the choices above, we have $d_{A}=d+\left[a\land\cdot\right]$,
where $a$ is a skew-Hermitian matrix of genuine $1$-forms extending smoothly over the boundary. Equivalently, they
are smooth $O\left(\rho\right)$ $0$-$1$-forms. Therefore, $\left[a\land\cdot\right]$
does not contribute to the normal operator $N_{p}\left(d_{A}\right)$,
which is just the direct sum of $r^{2}$ copies of the exterior derivative
$d:{^{0}\Omega^{k}}\to{^{0}\Omega^{k+1}}$ on $\overline{\mathbb{H}}^{n+1}$. It then follows from Proposition \ref{prop: composition properties of normal ops} that the normal operator of the twisted Hodge Laplacian $d_A^* d_A + d_A d_A^*$ on $\mathfrak{u}\left(\overline{E}\right)$-valued $0$-$k$-forms is just the direct sum of $r^2$ times the Hodge Laplacian on $0$-$k$-forms over $\overline{\mathbb{H}}^{n+1}$. Moreover, $F_A$ is a smooth $2$-form up to the boundary, so it can be interpreted as a $O\left(\rho^2\right)$ smooth $0$-$2$-form. Therefore, $\left[\cdot^{*}\contr F_{A}\right]$ does not contribute to the normal operator $N_p \left(L_A\right)$.
\end{proof}
\begin{rem}\label{rem: nondegenerate 0-connections}
Lemma \ref{lem: normal ops of L_A} is in general not true for a $0$-connection $A$.
The local $0$-$1$-forms representing $A$ with respect to a unitary
local frame of $\overline{E}$ near the boundary might not vanish along the boundary, giving a contribution to the normal operators. This complicates the analysis considerably.
\end{rem}

Lemma \ref{lem: normal ops of L_A} implies that the indicial roots of $\Delta_{A}$ (resp. $L_{A}$) do not depend
on $p$, and coincide with the indicial roots of the Hodge Laplacian $\Delta_{0}$ (resp.
$\Delta_{1}$). These  were computed by Mazzeo in \cite{MazzeoHodge,MazzeoPhd}:

\begin{lem}
(Mazzeo) The indicial roots of the Hodge Laplacian $\Delta_{k}$ on
$0$-$k$-forms over $\overline{\mathbb{H}}^{n+1}$ are
\[
\begin{cases}
0,n & k=0,n+1\\
k,n-k,k-1,n-\left(k-1\right) & k\not=0,n+1.
\end{cases}
\]
\end{lem}

It is proved in \cite{DonnellyXavier} that, for $k < n/2$ (and hence for $k> \left(n+2\right)/2$ by duality) the $L^2_0$ kernel of $ \Delta_k $ on $\overline{\mathbb{H}}^{n+1}$  vanishes. This result is generalized in \cite{MazzeoHodge}, where the  $L^2_0$ kernel of $ \Delta_k $ on a conformally compact manifold is given a topological interpretation:

\begin{thm} (Mazzeo)
Let $\left(\overline{M}^{n+1},g\right)$ be an oriented conformally compact manifold. Assume that $k< n/2$. Then the $L^2_0$ kernel of $\Delta_k$ is isomorphic to $H^k\left( \overline{M},X\right)$.
\end{thm}

We can now apply Theorem \ref{thm: Mazzeo's thm self-adjoint} to $\Delta_A$ and $L_A$:
\begin{thm}
\label{thm: Fredholmness of laplacians}Let $k\in\mathbb{N},k\geq2$
and $\alpha\in\left(0,1\right)$. Then:
\begin{enumerate}
\item for any $\delta\in\left(0,n\right)$,
\begin{align}
\Delta_{A} & :\rho^{\delta}C_{0}^{k,\alpha}\left(\mathfrak{u}\left(\overline{E}\right)\right)\to\rho^{\delta}C_{0}^{k-2,\alpha}\left(\mathfrak{u}\left(\overline{E}\right)\right)\label{eq: Laplacian on u(E)}
\end{align}
is Fredholm of index zero;
\item for any $\delta\in\left(1,n-1\right)$,
\begin{equation}
L_{A}:\rho^{\delta}C_{0}^{k,\alpha}\left(^{0}\Lambda_{\mathfrak{u}\left(\overline{E}\right)}^{1}\right)\to\rho^{\delta}C_{0}^{k-2,\alpha}\left(^{0}\Lambda_{\mathfrak{u}\left(\overline{E}\right)}^{1}\right)\label{eq: YM Laplacian}
\end{equation}
is Fredholm of index zero. Moreover, its kernel coincides with the $L_{0}^{2}$ kernel.
\end{enumerate}
\end{thm}

\begin{rem}\label{rem: why n+1>=4}The hypothesis $n+1\geq 4$ on the dimension of $\overline{M}$ is necessary to ensure that there is at least one "Fredholm weight" for $L_A$.
\end{rem}

We can say more about $\Delta_{A}$:
\begin{cor}
\label{cor: invertibility of laplacian on functions}For any $k\in\mathbb{N},k\geq2$,
$\alpha\in\left(0,1\right)$, and $\delta\in\left(0,n\right)$, the
map \ref{eq: Laplacian on u(E)} is invertible.
\end{cor}

\begin{proof}
This is a straightforward application of the maximum principle.
Let $\Delta=d^* d$ be the Laplacian on functions, and let $u$ be in the
$\rho^{\delta}C_{0}^{k,\alpha}$ kernel of $\Delta_{A}$. Since $k\geq2$,
$u$ is $C^{2}$ in the interior. Since $A$ is Hermitian, we have
\[
-\frac{1}{2}\Delta\left|u\right|^{2}=\left\langle d_{A}u,d_{A}u\right\rangle \geq0.
\]
It follows that $\left|u\right|^{2}$ cannot have a global maximum
in $M$ unless it is constant. However, since $\delta>0$, Lemma \ref{lem: (the useful lemma)}
implies that $u\in C^{0}\left(\overline{M}\right)$ and vanishes along
$X$. Since $\left|u\right|^{2}$ is continuous on $\overline{M}$,
it has a global maximum, that must be attained at $X$. It follows
that $u\equiv0$.
\end{proof}

\subsection{Gauge fixing\label{subsec:Gauge-fixing}}

Let's choose $k\in\mathbb{N},k\geq3$, $\alpha\in\left(0,1\right)$,
and $\delta\in\left(1,2\right)$. In this subsection we will define:
\begin{enumerate}
\item a Banach manifold $\mathcal{A}^{k,\alpha}\left(X\right)$ of connections
on $\overline{E}_{|X}$;
\item a Banach manifold $\mathcal{A}_{\delta}^{k,\alpha}$ of connections
on $\overline{E}$ whose restriction to the boundary is a well defined
element of $\mathcal{A}^{k,\alpha}\left(X\right)$;
\item a Banach Lie group $\mathcal{G}_{\delta}^{k+1,\alpha}$ of sections
of $\text{U}\left(\overline{E}\right)$ which acts on $\mathcal{A}_{\delta}^{k,\alpha}$
by pullback and preserves the boundary connection;
\item a subset $\mathcal{S}_{A_{0}}\subseteq\mathcal{A}_{\delta}^{k,\alpha}$
which provides a ``local slice'' for the action of $\mathcal{G}_{\delta}^{k+1,\alpha}$
on $\mathcal{A}_{\delta}^{k,\alpha}$ at a smooth connection $A_{0}$ (cf.\ Proposition
\ref{prop: slice} for a more precise statement).
\end{enumerate}
We will follow a strategy similar to the one
outlined in \cite{MazzeoPacard} in the context of Poincaré--Einstein
metrics.

$\mathcal{A}^{k,\alpha}\left(X\right)$ and $\mathcal{G}_{\delta}^{k+1,\alpha}$
are easy to define:
\begin{defn}
We define $\mathcal{A}^{k,\alpha}\left(X\right)=\omega_{0}+C^{k,\alpha}\left(T^{*}X\otimes\mathfrak{u}\left(\overline{E}_{|X}\right)\right)$
where $\omega_{0}$ is a smooth connection on $\overline{E}_{|X}$.
\end{defn}

$\mathcal{A}^{k,\alpha}\left(X\right)$ is an affine space modelled
on $C^{k,\alpha}\left(T^{*}X\otimes\mathfrak{u}\left(\overline{E}_{|X}\right)\right)$,
and its Banach manifold structure does not depend on $\omega_{0}$.
\begin{defn}
We define $\mathcal{G}_{\delta}^{k+1,\alpha}$ to be the group of
continuous sections of $\text{U}\left(\overline{E}\right)$ contained in $1+\rho^{\delta}C_{0}^{k+1,\alpha}\left(\mathfrak{gl}\left(\overline{E}\right)\right)$.
\end{defn}

Here a few comments are necessary. First of all, since $\delta>0$,
Lemma \ref{lem: (the useful lemma)} implies that
\[
\rho^{\delta}C_{0}^{k+1,\alpha}\left(\mathfrak{gl}\left(\overline{E}\right)\right)\hookrightarrow C^{0}\left(\mathfrak{gl}\left(\overline{E}\right)\right)
\]
and that its elements vanish along $X$. Therefore, every gauge transformation
in $\mathcal{G}_{\delta}^{k+1,\alpha}$ restricts to the identity
along $X$. Now, consider the map
\begin{align*}
1+\rho^{\delta}C_{0}^{k+1,\alpha}\left(\mathfrak{gl}\left(\overline{E}\right)\right) & \to\rho^{\delta}C_{0}^{k+1,\alpha}\left(\text{H}\left(\overline{E}\right)\right),\\
\Phi=1+u & \mapsto\Phi^{*}\Phi-1=u^{*}+u+u^{*}u,
\end{align*}
where $\text{H}\left(\overline{E}\right)$ is the real bundle of Hermitian
endomorphisms of $\overline{E}$. $\mathcal{G}_{\delta}^{k+1,\alpha}$
is the zero locus of this map, and Lemma \ref{lem: (the useful lemma)}
implies that this map is well defined and smooth. Moreover, $0$ is
a regular value, so $\mathcal{G}_{\delta}^{k+1,\alpha}$ inherits
from $1+\rho^{\delta}C_{0}^{k+1,\alpha}\left(\mathfrak{gl}\left(\overline{E}\right)\right)$
a manifold structure. Finally, again by Lemma \ref{lem: (the useful lemma)},
the multiplication and inverse maps on $\mathcal{G}_{\delta}^{k+1,\alpha}$
are well defined and smooth, so $\mathcal{G}_{\delta}^{k+1,\alpha}$
is a Banach Lie group.

In order to define our space $\mathcal{A}_{\delta}^{k,\alpha}$ of
connections on $\overline{E}$, we need some more auxiliary data.
Let us fix once and for all a special boundary defining function $\rho$
inducing a metric $h_{0}=\left(\rho^{2}g\right)_{|X}$ on $X$. Recall
from Subsection \ref{subsec:Conformally-compact-manifolds} that $\rho$
induces a collar $X\times[0,\varepsilon)\hookrightarrow\overline{M}$,
where $\varepsilon>0$ is chosen small enough so that $\left|d\rho/\rho\right|_{g}\equiv1$
in the collar. Let us fix a cutoff function $\chi:[0,+\infty)\to\left[0,1\right]$,
equal to $1$ near $0$ and supported in $[0,\varepsilon)$.

Since $\rho^{2}g=d\rho^{2}+h\left(\rho\right)$ on the collar, the
restriction of the bundle $\Lambda^{1}$ to the boundary $X$ splits
$\rho^{2}g$-orthogonally as
\[
\Lambda_{|X}^{1}=\text{span}\left(d\rho_{|X}\right)\oplus T^{*}X.
\]
On the collar induced by $\rho$, we can always decompose uniquely
a smooth section $a$ of $\Lambda^{1}$ as
\[
a=a_{t}\left(\rho\right)d\rho+a_{s}\left(\rho\right),
\]
where $a_{t}\left(\rho\right)$ is interpreted as a smooth path $[0,\varepsilon)\to C^{\infty}\left(X\right)$,
and $a_{s}\left(\rho\right)$ as a smooth path \mbox{$[0,\varepsilon)\to\Omega^{1}\left(X\right)$}.

Now, choose a reference smooth connection $A_0$. The parallel transport of $A_{0}$ along the integral curves of $\text{grad}_{\rho^{2}g}\rho$
induces bundle isometries $\overline{E}_{|X}\to\overline{E}_{|\rho^{-1} \left(\lambda\right)}$ for every $\lambda \in \left(0,\varepsilon\right)$.
Using this identification, we can always decompose a section $a$
of $\Lambda_{\mathfrak{u}\left(\overline{E}\right)}^{1}$ on the collar
as
\[
a=a_{t}\left(\rho\right)d\rho+a_{s}\left(\rho\right),
\]
where now $a_{t}:[0,\varepsilon)\to C^{\infty}\left(\mathfrak{u}\left(\overline{E}_{|X}\right)\right)$
and $a_{s}:[0,\varepsilon)\to C^{\infty}\left(T^{*}X\otimes\mathfrak{u}\left(\overline{E}_{|X}\right)\right)$.

Keeping in mind those interpretations, we can define a linear bounded
extension map
\begin{align*}
e:C^{k,\alpha}\left(T^{*}X_{\mathfrak{u}\left(\overline{E}_{|X}\right)}\right) & \to C^{k,\alpha}\left(\Lambda_{\mathfrak{u}\left(\overline{E}\right)}^{1}\right).\\
\gamma & \mapsto\chi\left(\rho\right)\gamma
\end{align*}
We are finally ready to define $\mathcal{A}_{\delta}^{k,\alpha}$.
\begin{defn}
\label{def: space of connections}We define $\mathcal{A}_{\delta}^{k,\alpha}$
as the space of connections of the form $A_{0}+e\left(\gamma\right)+a$,
where 
\begin{align*}
\gamma & \in C^{k,\alpha}\left(T^{*}X_{\mathfrak{u}\left(\overline{E}_{|X}\right)}\right)\\
a & \in\rho^{\delta}C_{0}^{k,\alpha}\left({^{0}\Lambda^{1}}_{\mathfrak{u}\left(\overline{E}\right)}\right).
\end{align*}
\end{defn}

Again, a few remarks are necessary. First of all, by Lemma \ref{lem: (the useful lemma)}
and since $\delta>1$, we have bounded inclusions
\begin{align*}
C^{k,\alpha}\left(\Lambda_{\mathfrak{u}\left(\overline{E}\right)}^{1}\right) & \hookrightarrow\rho C_{0}^{k,\alpha}\left({^{0}\Lambda^{1}}_{\mathfrak{u}\left(\overline{E}\right)}\right) \\
\rho^\delta C_{0}^{k,\alpha}\left({^{0}\Lambda^{1}}_{\mathfrak{u}\left(\overline{E}\right)}\right) & \hookrightarrow\rho C_{0}^{k,\alpha}\left({^{0}\Lambda^{1}}_{\mathfrak{u}\left(\overline{E}\right)}\right)
,
\end{align*}
so the sum $e\left(\gamma\right)+a$ makes sense inside the ambient
space $\rho C_{0}^{k,\alpha}$.

Now, again by Lemma\textbf{ }\ref{lem: (the useful lemma)},
we have
\[
\rho^{\delta}C_{0}^{k,\alpha}\left(^{0}\Lambda_{\mathfrak{u}\left(\overline{E}\right)}^{1}\right)=\rho^{\delta-1}C_{0}^{k,\alpha}\left(\Lambda_{\mathfrak{u}\left(\overline{E}\right)}^{1}\right)\hookrightarrow C^{0}\left(\Lambda_{\mathfrak{u}\left(\overline{E}\right)}^{1}\right)
\]
(here the metric on $\Lambda^{1}$ is the one induced by $\rho^{2}g$) and its\textbf{ }elements
vanish along $X$. Therefore, for every $A=A_{0}+e\left(\gamma\right)+a\in\mathcal{A}_{\delta}^{k,\alpha}$,
we have a well defined restriction $A_{|X}$ and
\[
A_{|X}=A_{0|X}+\gamma\in\mathcal{A}^{k,\alpha}\left(X\right).
\]
From this it follows also that, if $A,A'\in\mathcal{A}_{\delta}^{k,\alpha}$,
say $A=A_{0}+e\left(\gamma\right)+a$ and $A'=A_{0}+e\left(\gamma'\right)+a'$,
then
\[
A=A'\iff\gamma=\gamma'\,\text{and}\,a=a'.
\]
In other words, $\mathcal{A}_{\delta}^{k,\alpha}$ is modelled on
the direct sum
\[
C^{k,\alpha}\left(T^{*}X_{\mathfrak{u}\left(\overline{E}\right)}\right)\oplus\rho^{\delta}C_{0}^{k,\alpha}\left(^{0}\Lambda_{\mathfrak{u}\left(\overline{E}\right)}^{1}\right).
\]

Finally, $\mathcal{A}_{\delta}^{k,\alpha}$ \emph{does} depend, even
as a set, on the choice of the reference smooth connection $A_{0}$.
In fact, $\mathcal{A}_{\delta}^{k,\alpha}$ does not contain
every smooth connection. More precisely, if $A$ is another smooth
connection, then $\left(A-A_{0}\right)_{|X}$ is a smooth section
of $\Lambda_{|X}^{1}=\text{span}\left(d\rho_{|X}\right)\oplus T^{*}X$
with values in $\mathfrak{u}\left(\overline{E}_{|X}\right)$, and
$A\in\mathcal{A}_{\delta}^{k,\alpha}$ if and only if its $d\rho_{|X}$ part
vanishes along $X$.

\begin{lem}
\label{lem: smoothness of action}The pullback action of $\mathcal{G}_{\delta}^{k+1,\alpha}$
on $\mathcal{A}_{\delta}^{k,\alpha}$ is well defined and smooth,
and preserves the boundary map $\mathcal{A}_{\delta}^{k,\alpha}\to\mathcal{A}^{k,\alpha}\left(X\right)$,
$A\mapsto A_{|X}$.
\end{lem}

\begin{proof}
Let $\Phi\in\mathcal{G}_{\delta}^{k+1,\alpha}$ and $A\in\mathcal{A}_{\delta}^{k,\alpha}$.
We can write $\Phi=1+u$ for some $u\in\rho^{\delta}C_{0}^{k+1,\alpha}\left(\mathfrak{gl}\left(\overline{E}\right)\right)$,
and $A=A_{0}+e\left(\gamma\right)+a$ as in Definition \ref{def: space of connections}.
Write $\tilde{a}=e\left(\gamma\right)+a$. The pullback of $A$ by
$\Phi$ is 
\begin{align}
A\cdot\Phi & =A+\Phi^{*}d_{A}\Phi\nonumber \\
 & =A_{0}+e\left(\gamma\right)\label{eq: gauge in coordinates}\\
 & +a+\left(1+u^{*}\right)\left(d_{A_{0}}u+\left[\tilde{a}\land u\right]\right).\nonumber 
\end{align}
The discussion after Definition \ref{def: space of connections} shows
that $\tilde{a}=\tilde{a}\left(\gamma,a\right)$ is bilinear continuous
(hence smooth) as a map with values in $\rho C_{0}^{k,\alpha}\left(^{0}\Lambda^{1}\otimes\mathfrak{u}\left(\overline{E}\right)\right)$.
Therefore, by Lemma \ref{lem: (the useful lemma)}, the last row in
the formula above is a $\rho^{\delta}C_{0}^{k,\alpha}\left(^{0}\Lambda^{1}\otimes\mathfrak{u}\left(\overline{E}\right)\right)$-valued
smooth function of $\gamma,a$. This shows that the pullback map is
well defined and smooth. Moreover, we already noticed that 
\[
\rho^{\delta}C_{0}^{k,\alpha}\left(^{0}\Lambda_{\mathfrak{u}\left(\overline{E}\right)}^{1}\right)\hookrightarrow C^{0}\left(\Lambda_{\mathfrak{u}\left(\overline{E}\right)}^{1}\right)
\]
and its elements vanish along the boundary. Therefore, the $1$-form
in the last row of \ref{eq: gauge in coordinates} vanishes along
$X$, and hence $\left(A\cdot\Phi\right)_{|X}=A_{|X}=A_{0|X}+\gamma$.
\end{proof}
Now we will define a ``slice'' for the action at the point $A_{0}$.
Define
\begin{equation}
\mathcal{S}_{A_{0}}=\left\{ A_{0}+e\left(\gamma\right)+a\in\mathcal{A}_{\delta}^{k,\alpha}:d_{A_{0}+e\left(\gamma\right)}^{*}a=0\right\} .\label{eq: slice}
\end{equation}
In other words, $A\in\mathcal{S}_{A_{0}}$ \emph{if and only if it
is in Coulomb gauge with $A_{0}+e\left(\left(A-A_{0}\right)_{|X}\right)$}.
Endow $\mathcal{S}_{A_{0}}$ with the topology induced by $\mathcal{A}_{\delta}^{k,\alpha}$.
\begin{lem}
Near $A_{0}$, $\mathcal{S}_{A_{0}}$ is a smooth Banach submanifold
of $\mathcal{A}_{\delta}^{k,\alpha}$, with tangent space at $A_{0}$
given by
\[
T_{A_{0}}\left(\mathcal{S}_{A_{0}}\right)=C^{k,\alpha}\left(T^{*}X_{\mathfrak{u}\left(\overline{E}_{|X}\right)}\right)\oplus\ker_{\rho^{\delta}C_{0}^{k,\alpha}\left({^{0}\Lambda_{\mathfrak{u}\left(E\right)}^{1}}\right)}d_{A_{0}}^{*}.
\]
\end{lem}

\begin{proof}
$\mathcal{S}_{A_{0}}$ is the zero locus of the map
\begin{align*}
\mathcal{A}_{\delta}^{k,\alpha} & \to\rho^{\delta}C_{0}^{k-1,\alpha}\left(\mathfrak{u}\left(\overline{E}\right)\right)\\
A_{0}+e\left(\gamma\right)+a & \mapsto d_{A_{0}+e\left(\gamma\right)}^{*}a.
\end{align*}
This map is well defined and smooth because $d_{A_0}$ is a $0$-differential operator,
\[
e:C^{k,\alpha}\left(T^{*}X_{\mathfrak{u}\left(\overline{E}_{|X}\right)}\right)\to\rho C_{0}^{k,\alpha}\left(^{0}\Lambda_{\mathfrak{u}\left(\overline{E}\right)}^{1}\right)
\]
is bounded, and $\star\left[\cdot\land\star\cdot\right]:\rho C_{0}^{k,\alpha}\times\rho^{\delta}C_{0}^{k,\alpha}\to\rho^{\delta}C_{0}^{k,\alpha}$
is smooth by Lemma \ref{lem: (the useful lemma)}. The derivative
of this map at $A_{0}$ is
\begin{align*}
T_{A_{0}}\mathcal{A}_{\delta}^{k,\alpha} & \to\rho^{\delta}C_{0}^{k-1,\alpha}\left(\mathfrak{u}\left(\overline{E}\right)\right)\\
\left(\gamma,a\right) & \mapsto d_{A_{0}}^{*}a.
\end{align*}
Now, we know from Corollary \ref{cor: invertibility of laplacian on functions}
that
\[
d_{A_{0}}^{*}:\rho^{\delta}C_{0}^{k,\alpha}\left(^{0}\Lambda_{\mathfrak{u}\left(\overline{E}\right)}^{1}\right)\to\rho^{\delta}C_{0}^{k-1,\alpha}\left(\mathfrak{u}\left(\overline{E}\right)\right)
\]
is surjective. The claim then follows from the inverse function theorem.
\end{proof}
We are finally ready to prove that $\mathcal{S}_{A_{0}}$ indeed provides
a ``local slice'' for the action at $A_{0}$.
\begin{prop}
\label{prop: slice}There are opens
\begin{align*}
\mathcal{U}_{A_{0}}\subseteq\mathcal{S}_{A_{0}} & , A_{0}\in\mathcal{U}_{A_{0}} \\
\mathcal{V}_{1}\subseteq\mathcal{G}_{\delta}^{k+1,\alpha} & , 1\in\mathcal{V}_{1} \\
\mathcal{W}_{A_{0}}\subseteq\mathcal{A}_{\delta}^{k,\alpha} & , A_{0}\in\mathcal{W}_{A_{0}}
\end{align*}
with the following property: for every $A\in\mathcal{W}_{A_{0}}$,
there is a unique $\Phi\in\mathcal{V}_{1}$ such that $A\cdot\Phi\in\mathcal{U}_{A_{0}}$.
\end{prop}

\begin{proof}
Restrict the action map to
\begin{align*}
\mathcal{S}_{A_{0}}\times\mathcal{G}_{\delta}^{k+1,\alpha} & \to\mathcal{A}_{\delta}^{k,\alpha}\\
\left(B,\Phi\right) & \mapsto B\cdot\Phi.
\end{align*}
This map is smooth near $\left(A_{0},1\right)$, and its derivative
at that point is is
\begin{align*}
T_{A_{0}}\left(\mathcal{S}_{A_{0}}\right)\oplus T_{1}\left(\mathcal{G}_{\delta}^{k+1,\alpha}\right) & \to T_{A_{0}}\left(\mathcal{A}_{\delta}^{k,\alpha}\right)\\
\left(\gamma,a,u\right) & \mapsto\left(\gamma,d_{A_{0}}u+a\right).
\end{align*}
$\gamma$ is just a passenger here: to conclude by the inverse function
theorem, it is sufficient to prove that the map
\begin{align*}
\ker_{\rho^{\delta}C_{0}^{k,\alpha}\left({^{0}\Lambda_{\mathfrak{u}\left(E\right)}^{1}}\right)}d_{A_{0}}^{*}\oplus\rho^{\delta}C_{0}^{k+1,\alpha}\left(\mathfrak{u}\left(\overline{E}\right)\right) & \to\rho^{\delta}C_{0}^{k,\alpha}\left({^{0}\Lambda_{\mathfrak{u}\left(\overline{E}\right)}^{1}}\right)\\
\left(a,u\right) & \mapsto d_{A_{0}}u+a
\end{align*}
is invertible.

\emph{\underline{Injectivity.}}\textbf{ }Suppose that $d_{A_{0}}u+a=0$.
If we apply $d_{A_{0}}^{*}$ again, we get $\Delta_{A_{0}}u=0$, and
hence Corollary \ref{cor: invertibility of laplacian on functions}
implies that $u=0$; therefore, $a=0$ as well.

\emph{\underline{Surjectivity.}}\textbf{ }Let $b\in\rho^{\delta}C_{0}^{k,\alpha}\left(^{0}\Lambda^1 \otimes {\mathfrak{u}\left(\overline{E}\right)}\right)$.
Corollary \ref{cor: invertibility of laplacian on functions} ensures
that there is a unique solution $u$ in $\rho^{\delta}C_{0}^{k+1,\alpha}\left(\mathfrak{u}\left(\overline{E}\right)\right)$
of the equation $\Delta_{A_{0}}u=d_{A_{0}}^{*}b$. Define $a=b-d_{A_{0}}u$:
then $a$ is in \mbox{$\rho^{\delta}C_{0}^{k,\alpha}\left(^{0}\Lambda^1 \otimes {\mathfrak{u}\left(\overline{E}\right)}\right)$}
and $d_{A_{0}}^{*}a=d_{A_{0}}^{*}b-\Delta_{A_{0}}u=0$, so $\left(a,u\right)$
is in the domain of the map above.
\end{proof}

\subsection{A boundary value problem\label{subsec: a bvp}}

Let $A_{0}$ be a smooth Yang--Mills connection. We use $A_{0}$
as our reference connection to define the space $\mathcal{A}_{\delta}^{k,\alpha}$,
as in Subsection \ref{subsec:Gauge-fixing}. In this subsection we
will try and find solutions to the following gauge-fixed boundary
value problem for Yang--Mills connections
\begin{equation}
\begin{cases}
d_{A}^{*}F_{A}=0 & A\in\mathcal{A}_{\delta}^{k,\alpha}\\
A\in\mathcal{S}_{A_{0}}\\
A_{|X}=\omega & \omega\in\mathcal{A}^{k,\alpha}\left(X\right)
\end{cases}\label{eq: gauge-fixed bvp}
\end{equation}
for $A$ sufficiently close to $A_{0}$ and $\omega$ sufficiently close to $A_{0|X}$.

Consider the nonlinear map
\begin{align*}
\mathcal{Y}_{A_{0}}:\mathcal{A}_{\delta}^{k,\alpha} & \to\rho^{\delta}C_{0}^{k-2,\alpha}\left(^{0}\Lambda_{\mathfrak{u}\left(\overline{E}\right)}^{1}\right)\\
A=A_{0}+e\left(\gamma\right)+a & \mapsto d_{A}^{*}F_{A}+d_{A}d_{A_{0}+e\left(\gamma\right)}^{*}a.
\end{align*}

\begin{lem}
\label{lem: smoothness of Y_A_0}The map $\mathcal{Y}_{A_{0}}$ is
well defined and smooth.
\end{lem}

Before proving this lemma, let us introduce the $0$-differential
operators
\begin{align*}
P_{A_{0}} & =d_{A_{0}}^{*}d_{A_{0}}+\left[\cdot^{*}\contr F_{A_{0}}\right]\\
L_{A_{0}} & =d_{A_{0}}^{*}d_{A_{0}}+d_{A_{0}}d_{A_{0}}^{*}+\left[\cdot^{*}\contr F_{A_{0}}\right]
\end{align*}
on $\mathfrak{u}\left(\overline{E}\right)$-valued $0$-$1$-forms.
Note that $P_{A_{0}}$ is the linearization of $A\mapsto d_{A}^{*}F_{A}$,
and $L_{A_{0}}$ is the linearization of $A\mapsto d_{A}^{*}F_{A}+d_{A}d_{A_{0}}^{*}\left(A-A_{0}\right)$.
\begin{proof}
Write $\tilde{a}=e\left(\gamma\right)+a$. We can then rewrite $\mathcal{Y}_{A_{0}}$
as
\begin{align*}
\mathcal{Y}_{A_{0}}\left(A\right) & =P_{A_{0}}e\left(\gamma\right)+L_{A_{0}}a\\
 & +\frac{1}{2}d_{A_{0}}^{*}\left[\tilde{a}\land\tilde{a}\right]+\left[\tilde{a}^{*}\contr d_{A_{0}}\tilde{a}\right]+\frac{1}{2}\left[\tilde{a}^{*}\contr\left[\tilde{a}\land\tilde{a}\right]\right]\\
 & +d_{A_{0}}\left[e\left(\gamma\right)^{*}\contr a\right]+\left[\tilde{a}\land d_{A_{0}}^{*}a\right]+\left[\tilde{a}\land\left[e\left(\gamma\right)^{*}\contr a\right]\right].
\end{align*}
Arguing as in the proof of Lemma \ref{lem: smoothness of action},
we see that the last two rows are smooth
functions of $\gamma,a$ with values in $\rho^{\delta}C_{0}^{k-1,\alpha}\left(^{0}\Lambda^{1}\otimes\mathfrak{u}\left(\overline{E}\right)\right)$. Now, $L_{A_{0}}$ is a $0$-differential
operator, so $a\mapsto L_{A_{0}}a$ is bounded. Moreover, $\gamma\mapsto\left[e\left(\gamma\right)^{*}\contr F_{A_{0}}\right]$
is bounded by Lemma \ref{lem: (the useful lemma)}, because
\[
e:C^{k,\alpha}\left(T^{*}X_{\mathfrak{u}\left(\overline{E}_{|X}\right)}\right)\to C^{k,\alpha}\left(\Lambda_{\mathfrak{u}\left(\overline{E}\right)}^{1}\right)\hookrightarrow\rho C_{0}^{k,\alpha}\left(^{0}\Lambda_{\mathfrak{u}\left(\overline{E}\right)}^{1}\right)
\]
is bounded and, since $\delta<2$,
\[
F_{A_{0}}\in\Omega_{\mathfrak{u}\left(\overline{E}\right)}^{2}\subseteq\rho^{\delta}C_{0}^{k,\alpha}\left(^{0}\Lambda_{\mathfrak{u}\left(\overline{E}\right)}^{1}\right).
\]
It remains to prove that $\gamma\mapsto d_{A_{0}}^{*}d_{A_{0}}e\left(\gamma\right)$
is bounded. We can interpret
$
d_{A_{0}}$  as a differential operator between genuine  $\mathfrak{u}\left(\overline{E}\right)$-valued differential forms, so we get a bounded map
\[
d_{A_{0}}:C^{k,\alpha}\left(\Lambda_{\mathfrak{u}\left(\overline{E}\right)}^{1}\right)\to C^{k-1,\alpha}\left(\Lambda_{\mathfrak{u}\left(\overline{E}\right)}^{2}\right);
\]
since $\delta<2$, Lemma \ref{lem: (the useful lemma)} implies that
\[
C^{k-1,\alpha}\left(\Lambda_{\mathfrak{u}\left(\overline{E}\right)}^{2}\right)\hookrightarrow\rho^{\delta}C_{0}^{k-1,\alpha}\left(^{0}\Lambda_{\mathfrak{u}\left(\overline{E}\right)}^{2}\right),
\]
and since $d_{A_{0}}^{*}$ is a $0$-differential operator, we conclude
that $\gamma\mapsto d_{A_{0}}^{*}d_{A_{0}}e\left(\gamma\right)$ is
indeed bounded.
\end{proof}
\begin{rem}
This lemma is essentially the reason for which we chose $\delta<2$.
In fact, it is false for $\delta>2$, and although it is still true
for $\delta=2$, this case is problematic in dimension $n+1=4$ because
$2$ is an indicial root of $L_{A_{0}}$.
\end{rem}

The importance of the map $\mathcal{Y}_{A_{0}}$ for our purposes
is due to the following
\begin{prop}
\label{prop: nonlinear yang-mills}Let $A\in\mathcal{A}_{\delta}^{k,\alpha}$.
Then $\mathcal{Y}_{A_{0}}\left(A\right)=0$ if and only if $A$ is
Yang--Mills and belongs to the slice $\mathcal{S}_{A_{0}}$ defined
in Formula \ref{eq: slice}.
\end{prop}

\begin{proof}
The direction $\Leftarrow$ is automatic. Let's prove the direction
$\Rightarrow$. Write $A=A_{0}+e\left(\gamma\right)+a=A_{0}+\tilde{a}$
as above. Then
\[
d_{A}^{*}=d_{A_{0}}^{*}+\left[\tilde{a}^{*}\contr\cdot\right]:\rho^{\delta}C_{0}^{k-2,\alpha}\left(^{0}\Lambda_{\mathfrak{u}\left(\overline{E}\right)}^{1}\right)\to\rho^{\delta}C_{0}^{k-3,\alpha}\left(^{0}\Lambda_{\mathfrak{u}\left(\overline{E}\right)}^{1}\right)
\]
is bounded, because $d_{A_{0}}^{*}$ is a $0$-differential operator; moreover, since $\tilde{a}\in\rho C_{0}^{k,\alpha}$, by Lemma \ref{lem: (the useful lemma)}
$\left[\tilde{a}\contr\cdot\right]$ is bounded. Apply now $d_{A}^{*}$
to the equation $\mathcal{Y}_{A_{0}}\left(A\right)=0$: since \mbox{$d_A^* d_A^* F_A = \left(-1\right)^{n}\star\left[F_{A}\land\star F_{A}\right]=0$}, it follows that
\[
\Delta_{A}\left(d_{A_{0}+e\left(\gamma\right)}^{*}a\right)=0.
\]
Now, if $A$ were smooth, Corollary \ref{cor: invertibility of laplacian on functions} would imply that $d_{A_{0}+e\left(\gamma\right)}^{*}a=0$, i.e.\ that
$A\in\mathcal{S}_{A_{0}}$. However, notice that
\[
\Delta_{A}=\Delta_{A_{0}}+\left[\tilde{a}^{*}\contr d_{A_{0}}\cdot\right]+d_{A_{0}}^{*}\left[\tilde{a}\land\cdot\right]+\left[\tilde{a}^{*}\contr\left[\tilde{a}\land\cdot\right]\right].
\]
Since $\tilde{a}\in\rho C_{0}^{k,\alpha}$ and $d_{A_{0}},d_{A_{0}}^{*}$
are first order $0$-differential operators, Lemma \ref{lem: (the useful lemma)}
implies that the map
\[
\Delta_{A}-\Delta_{A_{0}}:\rho^{\delta}C_{0}^{k,\alpha}\to\rho^{\delta+1}C_{0}^{k-1,\alpha}\hookrightarrow\rho^{\delta}C_{0}^{k-2,\alpha}
\]
is \emph{compact}. Therefore, $\Delta_{A}$ is still Fredholm of index zero. Now, the proof of Corollary \ref{cor: invertibility of laplacian on functions}
works for our specific choice of $k$, so we can conclude that
\[
\Delta_{A}:\rho^{\delta}C_{0}^{k,\alpha} \left( \mathfrak{u}\left(\overline{E}\right)\right)\to\rho^{\delta}C_{0}^{k-2,\alpha}\left( \mathfrak{u}\left(\overline{E}\right)\right)
\]
is invertible.
\end{proof}
\begin{rem}
Here we used $k\geq3$ in order to make sense of the equations above classically.
\end{rem}

The last proposition shows that the zero locus of $\mathcal{Y}_{A_{0}}$
is exactly the set of solutions of the gauge-fixed boundary value
problem \ref{eq: gauge-fixed bvp}. We are now ready to prove the
main theorem:
\begin{thm}
\label{thm: main theorem}Let $A_{0}$ be a smooth Yang--Mills connection
on $\overline{E}$. Suppose that $L_{A_{0}}$ has trivial $L_{0}^{2}$
kernel. Then there are neighborhoods 
\begin{align*}
\mathcal{U} \subseteq C^{k,\alpha}\left(T^{*}X_{\mathfrak{u}\left(\overline{E}\right)}\right) & , 0\in \mathcal{U} \\
\mathcal{V} \subseteq \rho^{\delta}C_{0}^{k,\alpha}\left(^{0}\Lambda_{\mathfrak{u}\left(\overline{E}_{|X}\right)}^{1}\right) & , 0 \in \mathcal{V}
\end{align*}
and a unique smooth map $a:\mathcal{U}\to\mathcal{V}$ such that $a\left(0\right)=0$,
and for every $\gamma\in\mathcal{U}$, $A_{0}+e\left(\gamma\right)+a\left(\gamma\right)$
satisfies the boundary value problem \ref{eq: gauge-fixed bvp}.
\end{thm}

\begin{proof}
The computation in the proof of Lemma \ref{lem: smoothness of Y_A_0}
implies that
\begin{align*}
d\left(\mathcal{Y}_{A_{0}}\right)_{A_{0}}:T_{A_{0}}\mathcal{A}_{\delta}^{k,\alpha} & \to\rho^{\delta}C_{0}^{k-2,\alpha}\left(^{0}\Lambda_{\mathfrak{u}\left(\overline{E}\right)}^{1}\right)\\
\left(\gamma,a\right) & \mapsto P_{A_{0}}e\left(\gamma\right)+L_{A_{0}}a.
\end{align*}
The restriction of this linearization to the second component is just
\[
L_{A_{0}}:\rho^{\delta}C_{0}^{k,\alpha}\left(^{0}\Lambda_{\mathfrak{u}\left(\overline{E}\right)}^{1}\right)\to\rho^{\delta}C_{0}^{k-2,\alpha}\left(^{0}\Lambda_{\mathfrak{u}\left(\overline{E}\right)}^{1}\right),
\]
and by Theorem \ref{thm: Fredholmness of laplacians}, this map is
Fredholm of index zero, with kernel equal to the $L_{0}^{2}$ kernel. Since this kernel vanishes, the statement of the theorem follows immediately
from the implicit function theorem.
\end{proof}
Following the Poincaré--Einstein analogy with \cite{LeeFredholm},
we say that a smooth Yang--Mills connection $A_{0}$ is \emph{nondegenerate}
if the $L_{0}^{2}$ kernel of $L_{A_{0}}$ is trivial. Let's now give
examples of nondegenerate connections. The following corollary confirms
an expectation of Edward Witten in \cite{Witten}, and may be considered
as the Yang--Mills analogue of Graham and Lee's Theorem A in \cite{GrahamLee}.
\begin{cor}
\label{cor: Witten conjecture}Suppose that $\overline{M}$ satisfies $H^{1}\left(\overline{M},X\right)=0.$
Let $A_{0}$ be the trivial connection on the trivial bundle $\overline{E}=\overline{M}\times\mathbb{C}^{r}$
over $\overline{M}$. Then $A_{0}$ is nondegenerate, and hence Theorem
\ref{thm: main theorem} applies.
\end{cor}

\begin{proof}
Fixing a basis of $\mathfrak{u}\left(r\right)$, we get an identification
$\mathfrak{u}\left(\overline{E}\right)\equiv\overline{M}\times\mathbb{R}^{r^{2}}$.
Since $A_{0}$ is the trivial connection, its curvature vanishes
and $L_{A_{0}}$ is just the direct sum
of $r^{2}$ copies of the Hodge Laplacian $\Delta_{1}$ on $0$-$1$-forms.
Now, by Theorem 1.3 in \cite{MazzeoHodge}, the $L_{0}^{2}$ kernel
of $\Delta_{1}$ is isomorphic to $H^{1}\left(\overline{M},X\right)$,
and hence it vanishes by hypothesis. It follows that $A_0$ is nondegenerate.
\end{proof}
Observe that the converse argument gives examples of \emph{degenerate} Yang--Mills
connections: if $H^{1}\left(\overline{M},X\right)\not=0$, then
trivial connections are degenerate.



\bibliographystyle{abbrv}

\end{document}